\documentclass[a4paper]{amsart}

\usepackage{url,paralist,amssymb,amsthm,amsmath,stmaryrd,bbm}
\usepackage{hyperref}
\usepackage{psfrag}
\usepackage{pstricks}
\usepackage{overpic}
\usepackage[utf8]{inputenc}

\usepackage[small,bf]{caption2}  
\graphicspath{{./figures/}}

\newcommand\wedgeproduct{\operatorname{\leftslice}}

\newcommand\WPsurf[2]{\Sigma_{#1, #2}}
\newcommand\Aut[1]{\operatorname{\mathsf{Aut}}(#1)}

\newcommand\smallSetOf[2]{\{#1\colon #2\}}

\newcommand\simplex[1]{\Delta_{#1}}
\renewcommand\wp[2]{\mathcal{W}_{#1,#2}}

\newcommand\eps{\varepsilon}

\newcommand\moduli[1]{\mu(#1)}

\newcommand\R{\mathbb{R}}

\newcommand\Zq{[q]}
\newcommand\Zp{[p]}

\renewcommand\l{\ell}
\newcommand\1{\mathbf1}
\newcommand\zero{\mathbf{0}}
\newcommand\rowzero{\mathbbmss O}

\newtheorem{theorem}{Theorem}[section]
\newtheorem{corollary}[theorem]{Corollary}
\newtheorem{lemma}[theorem]{Lemma}
\newtheorem{proposition}[theorem]{Proposition}

\theoremstyle{definition}
\newtheorem{definition}[theorem]{Definition}

\theoremstyle{remark}
\newtheorem{remark}[theorem]{Remark}
\newtheorem*{remark*}{Remark}
\newtheorem*{observation*}{Observation}

\begin{document}
\title{Polyhedral Surfaces in Wedge Products}
\author[R\"orig \and Ziegler]{Thilo R\"orig \and G\"unter~M. Ziegler}
\address{Thilo R\"orig, Inst.\ Mathematics MA 8-3, TU Berlin, D-10623 Berlin, Germany}
\email{roerig@math.tu-berlin.de}
\address{G\"unter~M. Ziegler, Inst.\ Mathematics MA 6-2, TU Berlin, D-10623 Berlin, Germany}
\email{ziegler@math.tu-berlin.de}
\thanks{The authors are supported by Deutsche Forschungsgemeinschaft, 
via the DFG Research Group ``Polyhedral Surfaces'', and a Leibniz grant.}
\date{\today}

\begin{abstract}\noindent
  We introduce the wedge product of two polytopes. The wedge product is described in terms of inequality systems, in terms
  of vertex coordinates as well as purely combinatorially, from the corresponding data of
  its constituents.  The wedge product construction can be described as an iterated ``subdirect product'' as introduced
  by McMullen (1976); it is dual to the ``wreath product'' construction of Joswig and Lutz (2005).

  One particular instance of the wedge product construction turns out to be especially interesting: The wedge products
  of polygons with simplices contain certain combinatorially regular polyhedral surfaces as subcomplexes.  These
  generalize known classes of surfaces ``of unusually large genus'' that first appeared in works by Coxeter (1937),
  Ringel (1956), and McMullen, Schulz, and Wills (1983). Via ``projections of deformed wedge products'' we obtain
  realizations of some of the surfaces in the boundary complexes of $4$-polytopes, and thus in $\R^3$.  As additional
  benefits our construction also yields polyhedral subdivisions for the interior and the exterior, as well as a great
  number of local deformations (``moduli'') for the surfaces in~$\R^3$. In order to prove that there are many moduli,
  we introduce the concept of “affine support sets” in simple polytopes.  Finally, we explain how duality theory for
  $4$-dimensional polytopes can be exploited in order to also realize combinatorially dual surfaces in $\R^3$ via dual
  $4$-polytopes.
\end{abstract}
\maketitle 

\section{Introduction}
\label{sec:introduction}

In this paper we present a new instance of the fruitful interplay between the construction of ``interesting'' families
of polytopes and the construction of special polyhedral surfaces.  Two polytope constructions that previously turned out
to yield interesting polyhedral surfaces are the neighborly cubical polytopes (NCPs) of Joswig and
Ziegler~\cite{Joswig2000} and their generalization to ``projected deformed products of polygons'' (PDPPs) by
Ziegler~\cite{Ziegler2004a}. Both these families of $4$-polytopes are obtained by projections of high-dimensional simple
polytopes. In addition to other interesting properties, the $4$-dimensional NCPs contain surfaces ``of unusually large
genus'' first described by Coxeter~\cite{C:SkewPolyhedra} in terms of symmetry groups, then studied by
Ringel~\cite{R:kombinatorischeProblemeNWuerfel} and realized in $\R^3$ by McMullen, Schulz, and Wills~\cite{McMullen1983}.
The PDPPs also contain surfaces ``of unusually large genus'' that generalize the surfaces of \mbox{McMullen}, Schulz
and Wills (see also Rörig \cite{Thilo-phd} and Ziegler \cite{Z102a}).
In both cases the surfaces are contained in the boundary complexes of $4$-polytopes, which yields realizations of the
surfaces in $\R^3$ via Schlegel projection.

The situation presented and studied here is quite analogous to that of the surfaces obtained from NCP and PDPP
polytopes.  In the first part of this paper, we define the ``wedge product'' of two polytopes.  It is most conveniently
defined via explicit linear inequality systems.

In the second part of the paper, we analyze one particular instance, the wedge product $C_p\wedgeproduct\Delta_{q-1}$ of
a $p$-gon and a $(q-1)$-simplex.  In the $2$-skeleton of this wedge product we describe a \emph{regular} polyhedral
surface $\WPsurf{p}{2q}$: Its symmetry group is transitive on the flags. In particular, the surface
$\WPsurf{p}{2q}$ is ``equivelar'' of type $\{p,2q\}$ in the terminology of McMullen, Schulz, and Wills \cite{MSW:EPM}: It
is composed of convex $p$-gons, and each vertex has degree $2q$.  Thus we obtain the first geometric construction of
such surfaces for a wide range of parameters $\{p,2q\}$.  While our construction produces the surfaces in a
high-dimensional space, a general lemma of Perles~\cite[p.~204]{Grunbaum2003} yields realizations for these surfaces
in~$\R^5$.

We shall see that any wedge product of a $p$-gon and a simplex is simple. Thus it admits ``deformed'' realizations that
are quite analogous to the ``deformed product'' realizations of Amenta and Ziegler~\cite{AmentaZiegler1998}, Joswig and
Ziegler~\cite{Joswig2000}, and Sanyal and Ziegler~\cite{Ziegler2004a} \cite{Z102}.

For $q=2$, we can arrange things
such that the polyhedral surface $\WPsurf{p}{4}$ in the   
$2$-skeleton ``survives'' the projection to the boundary of a $4$-polytope.  Furthermore, the surface comes to lie on
the lower hull of the $4$-polytope, so it may be projected orthogonally to $\R^3$. (Our observations in this context
allow us to perform the projections to $3$-space without the usual use of Schlegel diagrams.) With the deformation
technique we also obtain a realization of the prism over the wedge product such that the prism over the surface
$\WPsurf{p}{4}$ is preserved by the projection to~$\R^4$. This allows us to use polytope duality to obtain realizations
for the dual surfaces $\WPsurf{p}{4}^*$, which are equivelar of type~$\{4,p\}$.

For $q>2$, R\"orig and Sanyal \cite{RoerigSanyal:nonproject} show that a suitable deformed realization of
$C_p\wedgeproduct\Delta_{q-1}$ in $\R^{2+p(q-1)}$ for which the surface $\WPsurf{p}{2q}$ realized in the boundary would
survive the projection to $\R^4$ does \emph{not} exist if $p>3$, due to topological reasons. The case $p=3$
remains open.

As an additional benefit of our construction we obtain that the realization spaces are high-dimensional (that is, that
there are ``many moduli'') near our realizations of the surfaces. As a proof technique for this we introduce 
and study “affine support sets” of simple polytopes.

\section{From wedges to wedge products}
\label{sec:Wedges2WedgeProducts}

In this section we first review the wedge construction, which proved to be useful in studies related to the Hirsch
conjecture (see e.g.\ Fritzsche and Holt~\cite{FritzscheHolt1999} and Santos and Kim~\cite{KimSantos_Hirsch09}). 
Then we extend this to obtain the generalized wedge
construction and derive some of its basic properties. By iterating the generalized wedge construction we obtain wedge
products.  We define these primarily in terms of linear inequalities, as this description is most suitable for our
projection purposes.

\subsection{Wedges and generalized wedges}
\label{sec:GeneralizedWedges}

Let $P$ be a $d$-dimensional polytope in~$\R^d$ with $m$ facets, given by its facet description $Ax \le \1$, for
$A\in\R^{m\times d}$, $x=(x_0,\dots,x_{d-1})^t$. Let $F$ be the facet of $P$ defined by the hyperplane $a_0 x = 1$. The
classical \emph{wedge} over a polytope $P$ at $F$ is constructed as follows: Embed $P\times\{0\}$ in $\R^{d+1}$ and
construct the cylinder $P \times \R \subset \R^{d+1}$. Then cut the cylinder with two disjoint hyperplanes through $F
\times \{0\}$ such that both cuts are bounded. These hyperplanes divide the cylinder into one bounded and two unbounded
components. The bounded part is the \emph{wedge}.  
It may be described by the inequality system
\[
\operatorname{wedge}_F(P)\ :=\ 
\Big\{
\begin{pmatrix}
  x \\ x_d
\end{pmatrix} \in \R^{d+1}\ \Big|\
  \left(
  \begin{array}{ccc}
    A'  &         \\ 
    a_0 & \pm 1
  \end{array}
  \right) 
  \begin{pmatrix}
    x \\ x_d
  \end{pmatrix}
  \le 
  \begin{pmatrix}
    \1 \\ 1
  \end{pmatrix}
\Big\},
\]
where $A'$ is the matrix $A$ with the row $a_0$ removed. The two hyperplanes that cut the cylinder are $a_0 x + x_d =
1$ and $a_0 x - x_d = 1$. They may be constructed by combining the equation $a_0x=1$ that defines the facet~$F$ with the
inequality description $\pm x_d \leq 1$ of the interval $[-1,+1]$ in $x_d$-direction.

Deletion of the last coordinates yields a projetion $\operatorname{wedge}_F(P) \rightarrow P$:
Fourier--Motzkin elimination of $x_d$ (that is, addition of the two inequalities involving $x_d$)
recovers $a_0 x\le 1$ as an inequality that is valid, but not facet-defining for $\operatorname{wedge}_F(P)$.

For the projection $\operatorname{wedge}_F(P) \rightarrow P$ the fiber above every point of $P$ is an interval~$I$,
except that it is a single point $\{*\}$ above every point of~$F$.  This might be indicated by
\[
(I,\{{*}\})\ \ \longrightarrow\ \ \operatorname{wedge}_F(P)\ \ \longrightarrow\ \ (P,F).
\]
For our purposes we need the following more general construction. 

\begin{definition}[generalized wedge $P \wedgeproduct_F Q$]
  \label{def:GeneralizedWedge}
  Let $P$ be a $d$-polytope in $\R^d$ with $m$ facets given by the inequality system $Ax \leq \1$, and let $Q$ be a
  $d'$-polytope in $\R^{d'}$ with $m'$ facets given by $By \leq \1$. Let $F$ be the facet of $P$ defined by the
  hyperplane~\mbox{$a_0x = 1$}.

  The \emph{generalized wedge} $P\wedgeproduct_F Q$ of $P$ and $Q$ at $F$ is the $(d+d')$-dimensional polytope defined
  by
  \begin{equation}
    \label{eq:generalizedWedge}
    P \wedgeproduct_F Q\ :=\
    \Big\{
    \begin{pmatrix}
      x \\ y
    \end{pmatrix} \in \R^{d+d'}\ \Big|\
    \left(
      \begin{array}{ccc}
        A'  &         \\ 
        A_0 & B
      \end{array}
    \right) 
    \begin{pmatrix}
      x \\ y
    \end{pmatrix}
    \le 
    \begin{pmatrix}
      \1 \\ \1
    \end{pmatrix}
    \Big\},
  \end{equation}
  where $A'$ is the matrix $A$ without the row $a_0$, and $A_0=\1 a_0$ is the $m' \times d$ matrix 
  all of whose rows are equal to~$a_0$.
\end{definition}

The generalized wedge $P\wedgeproduct_F Q$ is a $(d+d')$-dimensional polytope with $m-1+m'$ facets.  The classical wedge
$\operatorname{wedge}_F(P)$ may be viewed as the generalized wedge $P \wedgeproduct_F [-1,1]$.

\begin{proposition}\label{prop:GenWedgeProj}
  If $P$ and $Q$ are polytopes of dimension $d$ resp. $d'$, then the generalized wedge is a $(d+d')$-polytope $P
  \wedgeproduct_F Q$. It comes with a projection to $P$ (to the first $d$ coordinates) such that the fiber above every
  point of $P$ is an affine copy of~$Q$, except that it is a single point~$\{*\}$ above every point of~$F$. That is,
  \[
  (Q,\{{*}\})\ \ \longrightarrow\ \ P\wedgeproduct_F Q\ \ \longrightarrow\ \ (P,F).
  \]
\end{proposition}

\begin{proof}
First we show that the projection maps to $P$, that is, that the
inequality $a_0x \leq 1$ defining the facet $F$ of $P$ is valid (but not facet-defining) for the
the generalized wedge: Since $Q$ is bounded, its facet normals (the rows of $B$) are positively dependent,
so there is a positive row vector $c$ satisfying $cB=\rowzero$ and $c\1 =1$;
thus summing the inequalities in the system $A_0 x + B y\le \1$ with coefficients given by $c$ yields
\[
a_0x = cA_0 x + cB y\le c\1 = 1
\]
since $cA_0=c(\1 a_0)=(c\1)a_0= a_0$.

Now given any point $x\in P$, the fiber above $x$ is given by the system
$By\le \1 - A_0 x$. For $x\in F$ we have $A_0x=\1$, and $By\le\zero$ describes a point.
For $x\in P\setminus F$ we have $A_0x<\1$, and $By\le \1 - A_0 x$ describes
a copy of~$Q$ that has been scaled by a factor of $1-a_0x$.
This is schematically shown in Figure~\ref{fig:GeneralizedWedge}. 
\end{proof}

\begin{figure}[ht]
  \centering
  \begin{overpic}[width=.4\textwidth]{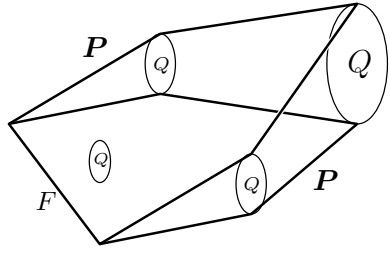}
    \put(8,10){$F$}
    \put(23,22){\tiny$Q$}
    \put(62,15){\scriptsize$Q$}
    \put(38.5,46.5){\scriptsize$Q$}
    \put(89,46){\Large$Q$}
    \put(80,15){\large{\boldmath$P$}}
    \put(20,50){\large{\boldmath$P$}}
  \end{overpic}
  \caption{This is a schematic drawing of the generalized wedge $P \wedgeproduct_F Q$. It is a degeneration of the
    product $P \times Q$ of two polytopes and contains many copies of both constituents.}
  \label{fig:GeneralizedWedge}
\end{figure}

\begin{remark}
  The subdirect product construction introduced by McMullen~\cite{McMullen1976} subsumes the generalized wedge 
  $P \wedgeproduct_F Q$ as the special case
 $(P,F)\otimes(Q,\emptyset)$.
\end{remark}

\begin{remark}
  The generalized wedge may be interpreted as a limit case (degeneration) of a deformed product in the sense of Amenta and
  Ziegler~\cite{AmentaZiegler1998}: If we consider an inequality $a_0x \leq 1+\eps$ for small $\eps > 0$ instead of the
  inequality $a_0x\leq 1$ defining the facet, then this inequality is strictly satisfied by all $x \in P$.  Further an
  inequality system similar to Equation~\eqref{eq:generalizedWedge} in Definition~\ref{def:GeneralizedWedge} defines a
  deformed product:
  \[
    \Big\{
    \begin{pmatrix}
      x \\ y
    \end{pmatrix} \in \R^{d+d'}\ \Big|\
    \left(
      \begin{array}{rcc}
        A_{\phantom0}       &         \\ 
        \frac{1}{1+\eps}A_0 & B
      \end{array}
    \right) 
    \begin{pmatrix}
      x \\ y
    \end{pmatrix}
    \le 
    \begin{pmatrix}
      \1 \\ \1
    \end{pmatrix}
    \Big\},
  \]
  where $A_0$ the $m' \times d$ matrix with all rows equal to $a_0$.  If $\eps \to 0$ then the facet $F \times Q$
  of the product $P \times Q$ degenerates to a lower dimensional face $F \times \{0\}$, and we obtain the generalized
  wedge. See Figure~\ref{fig:5gonGWs2} for an example.
  \begin{figure}[tbh]
    \centering
    \includegraphics[width=.3\textwidth]{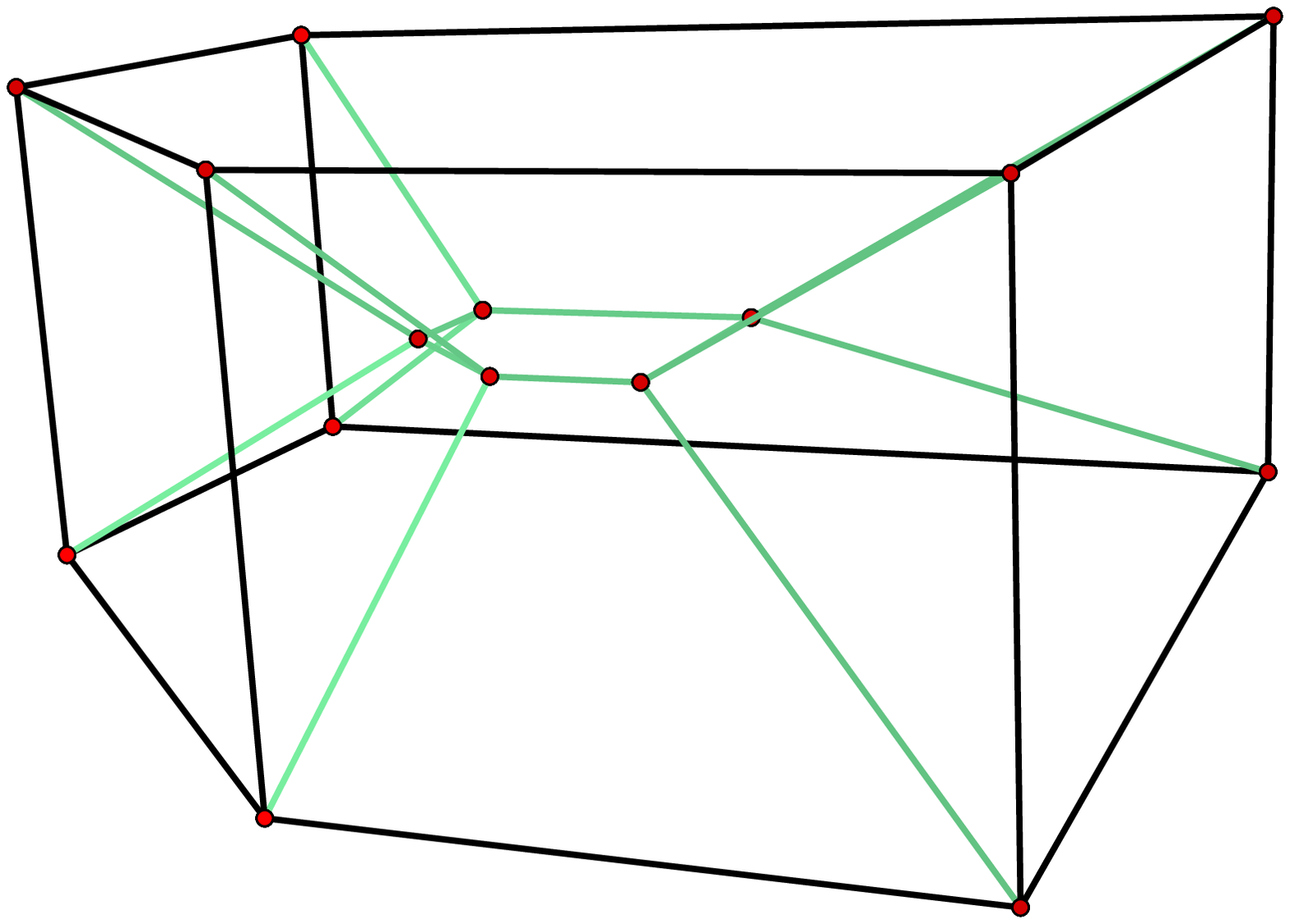}
    \includegraphics[width=.3\textwidth]{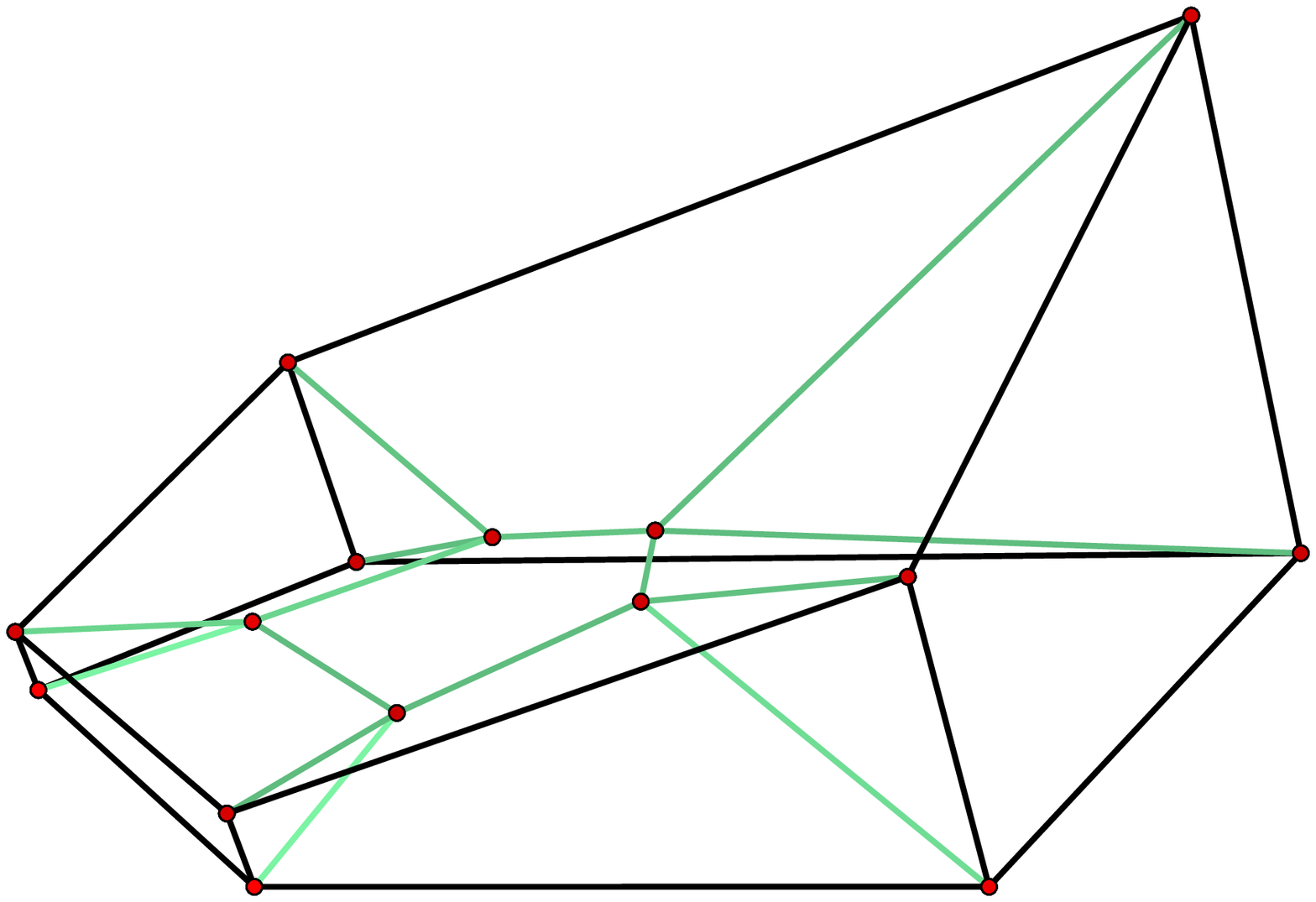}
    \includegraphics[width=.3\textwidth]{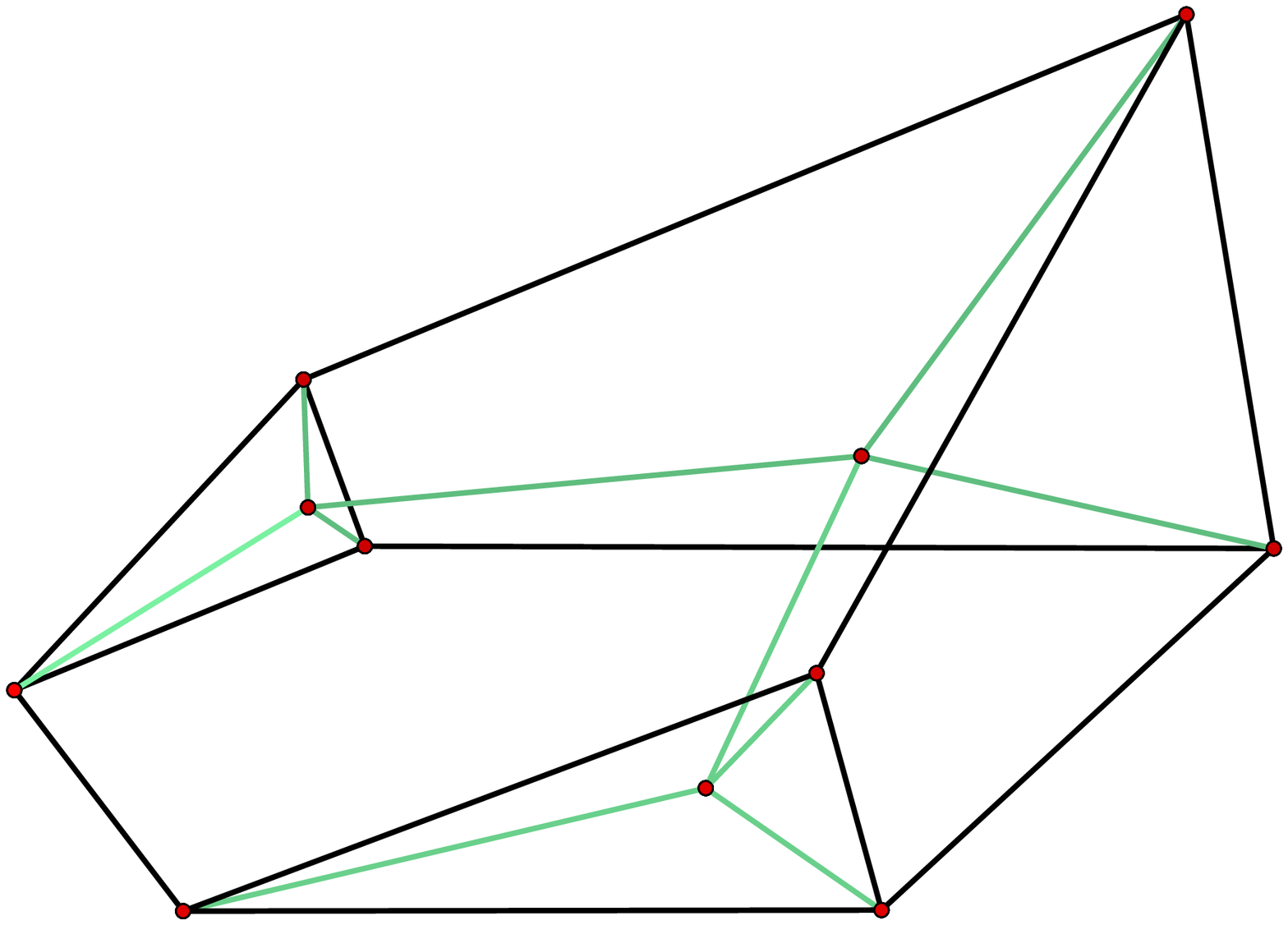}
    \caption{Schlegel diagrams showing the degeneration of a product to a generalized wedge: The orthogonal product of
      pentagon and triangle (left), a deformed product of pentagon and triangle (middle), and the generalized wedge of
      pentagon and triangle (right).}
    \label{fig:5gonGWs2}
  \end{figure}
\end{remark}

Using the degeneration of the deformed product to a generalized wedge or a little linear algebra we obtain the vertices
of the generalized wedge.

\begin{lemma}[Vertices of the generalized wedge $P \wedgeproduct_F Q$]\label{lem:VerticesOfGeneralizedWedge}
  Let $P \wedgeproduct_F Q$ be the generalized wedge of $P$ and $Q$ at $F$, 
  where $P$ has $n$ vertices and $Q$ has $n'$ vertices, and let $H = \smallSetOf{x \in \R^d}{a_0x =
    1}$ be a hyperplane defining the facet $F$ with $\bar n$ vertices. Then $P \wedgeproduct_F Q$ has 
 $( n     - \bar n) n'    +\bar n $ vertices.
 These belong to two families 
  \[
  u_{k\ell} = 
  \begin{cases}
    \tbinom{v_k}{0}                & \text{for }v_k \in    F, \quad 0\le k<n,\\
    \tbinom{v_k}{(1-a_0v_k) w_\ell}& \text{for }v_k \notin F, \quad 0\le k<n,\ 0\le \ell<n',
  \end{cases}
  \]
  where $v_k$ is a vertex of $P$ and $w_\ell$ is a vertex of $Q$. 
\end{lemma}

The generalized wedge $P \wedgeproduct_F Q$ contains many faces that are affinely equivalent to the ``base'' $P$.

\begin{proposition}
\label{prop:GenWedgePFace}
For an arbitrary vertex $w$ of $Q$ the
convex hull of the vertices $\tbinom{v_k}{(1-a_0v_k)w}$ for $k=0,\dots,n-1$ is a face that is affinely equivalent to $P$.
\end{proposition}

\begin{proof}
  The vertex $w\in Q$ is described by $\bar B y=\1$, where $\bar B$ is an invertible square matrix, and $\bar B y\le\1$
  is a subsystem of $By\le \1$. The corresponding subsystem $\bar A_0 x+\bar B y\le\1$ defines a face~$G_w$ of $P
  \wedgeproduct_F Q$, since it is a valid subsystem which is tight for the point $\tbinom{\zero}{w}$ that lies on the
  boundary of $P \wedgeproduct_F Q$. For any $x\in P$ we get a unique solution $y$ for $\bar A_0 x+\bar B y=\1$, which depends
  affinely
  on~$x$. Hence $x \mapsto (x,y)$ yields an affine equivalence between $P$ and the face~$G_w$ of~$P \wedgeproduct_F Q$
  that maps the vertices $v_k$ of $P$ to $\tbinom{v_k}{(1-a_0v_k)w}$ of~$G_w$.
\end{proof}

In rare cases the generalized wedge of two polytopes is simple. To characterize the simple wedge products we simply
count the facets at each vertex.
\begin{corollary}[Simple generalized wedges]
  \label{cor:SimpleGenWedges} 
  The generalized wedge $P \wedgeproduct_F Q$ of~$P$ and~$Q$ at~$F$ is simple if and only if
  \begin{itemize}[~$\bullet$]
  \item $P$ is a point and~$Q$ is simple (trivial case) or
  \item $P$ is simple and $Q$ is a simplex.
  \end{itemize}
\end{corollary}

\subsection{Wedge products}
\label{ssec:WedgeProducts}

The wedge product of two polytopes $P$ and $Q$ is obtained by iterating the generalized wedge construction for all
facet defining inequalities $a_ix=1$ of $P$. This is made explicit in the following definition.

\begin{definition}[wedge product $P \wedgeproduct Q$]
  \label{def:WedgeProduct} 
  Let $P$ be a $d$-polytope in~$\R^d$ given by $Ax\le\1$ with $m$ facets defined by $a_i x \leq 1$ for $i\in[m]$ and let
  $Q$ be a $d'$-polytope in~$\R^{d'}$ given by $By \le \1$ with $m'$ facets that are given by $b_jy \leq 1$ for
  $j\in[m']$.  For $i \in [m]$ denote by $A_i$ the $(m' \times d)$-matrix $\1 a_i$ with rows equal to $a_i$. The
  \emph{wedge product} $P \wedgeproduct Q$ is defined by the following system of inequalities:
  \begin{equation}
    \label{eq:wedgeProduct}
    P \wedgeproduct Q = 
    \Big\{
    \begin{pmatrix}
      x \\ y_0 \\ y_1 \\ \vdots \\ y_{m-1}
    \end{pmatrix}\in \R^{d+md'}
    \renewcommand{\arraystretch}{1}
    \ \Big|\ 
    \left(
      \arraycolsep2pt
      \begin{array}{ccccc}
        A_0    &  B   &     &       &      \\
        A_1    &      &  B  &       &      \\
        \vdots &      &     &\ddots &      \\
        A_{m-1}&      &     &       &   B  
      \end{array}
    \right)
    \begin{pmatrix}
      x \\ y_0 \\ y_1 \\ \vdots \\ y_{m-1}
    \end{pmatrix}
    \le
    \begin{pmatrix}
      \1 \\ \1 \\ \vdots \\ \1
    \end{pmatrix}
    \Big\}.
  \end{equation}
  We denote the hyperplanes $a_i x + b_j y_i = 1$ defining the wedge product by $h_{i,j}$ with $(i,j) \in
  [m]\times[m']$.
\end{definition}

\begin{remark}
  Comparing the inequality description of the wedge product to the vertex description of the wreath products of Joswig
  and Lutz~\cite{JoswigLutz2005}, we observe that wedge product and wreath product are dual constructions. In
  other words, if $P$ and $Q$ are polytopes and $P^*$ and $Q^*$ their duals, then the wedge product $P \wedgeproduct Q$
  is the dual of the wreath product $Q^* \wr P^*$.
\end{remark}

\begin{proposition}
  \label{prop:WPFiber}
  The wedge product $P \wedgeproduct Q$ of $P$ and $Q$ is a $(d+md')$-dimensional polytope with~$mm'$ facets~$h_{i,j}$
  indexed by $i\in[m]$ and $j\in[m']$.  It comes with a linear projection $P \wedgeproduct Q \rightarrow P$ (to the
  first $d$ coordinates).  The fiber above every interior point of $P$ is a product~$Q^m$, while the $i$-th factor in
  the fiber degenerates to a point above every point of~$P$ that is contained in the $i$-th facet of~$P$.
\end{proposition}

\begin{proof}
  This is analogous to the proof of Proposition~\ref{prop:GenWedgeProj}.
\end{proof}

\noindent In a ``fiber bundle'' interpretation, the situation might be denoted as
\[
(Q,\{{*}\})^m\ \ \longrightarrow\ \ P \wedgeproduct Q\ \ \longrightarrow\ \ (P,\{F_i\}_i).
\]
This picture has an analogy to MacPherson's topological description of the moment map $\mathbf{T}(P)\rightarrow P$ for a
toric variety, as presented in \cite{FischliYavin} and in \cite[Sect.~2.8]{MacPherson2004}.

We now give a purely combinatorial description of the faces of the wedge product. Any face $F$ of an arbitrary polytope
with $m$ facets $F_i$ ($i\in [m]$) may be identified with the set $\{ i\in[m]\ |\ F \subset F_i\}$ of 
indices of the facets that contain the face $F$. So we say that a set of facets $S \subset [m]$ ``is a face''
if and only if $S = \{ i\in[m]\ |\ F \subset F_i\}$ for some face~$F$. For the
wedge product $P \wedgeproduct Q$ each face is determined by a set of facets $h_{i,j}$, that is, 
it corresponds to a subset of
$[m]\times[m']$. Ordering this subset by the first index~$i$ the faces may be identified with a vector
$(H_0,\dots,H_{m-1})$ with $H_i \subseteq [m']$ in the following way:
\begin{equation}
  j \in H_i\ \Longleftrightarrow\ \text{$F$ lies on $h_{i,j}$}.\label{eq:FaceNotationOfWedgeProduct}
\end{equation}
In this correspondence, the facets of the wedge correspond to the $mm'$ ``unit coordinate vectors'' with one entry~$1$
and all other coordinates equal to~$0$. The vertices of a simple polytope $P$ correspond to vectors
with $\dim(P)$ ones and zeros otherwise.
With this notation we now describe the vertex facet incidences of the wedge product.

\begin{theorem}
  \label{prop:CombWedgeProducts}
  Let $P \wedgeproduct Q$ be the wedge product of polytopes $P$ and $Q$ with $m$ resp.~$m'$ facets. 
  Then $(H_0,\dots,H_{m-1})$ with $H_i \subseteq [m']$ corresponds to a vertex of $P \wedgeproduct Q$ if and only if:
  \begin{itemize}
  \item $\{ i \in [m]\ |\ H_i = [m'] \} \subseteq [m]$ corresponds to a vertex of $P$, and
  \item each $H_i$ with $H_i \neq [m']$ corresponds to a vertex of $Q$.
  \end{itemize}
\end{theorem}

Certain faces in a wedge product $P \wedgeproduct Q$ that are affinely equivalent to~$P$ will be particularly
interesting to us.

\begin{proposition}
  \label{prop:SpecialFacesWedgeProducts}
  Let $(H_0,\dots,H_{m-1})$ with $H_i \subseteq [m']$ correspond to a face $F$ of the wedge product $P \wedgeproduct Q$.
  If the intersection of the facets $b_j y \le 1$ ($j \in H_i$) is a vertex $w_i$ of $Q$ for all $i \in [m]$, then $F$
  is affinely equivalent to $P$.
\end{proposition}

\begin{proof}
  Every $H_i$ gives rise to a submatrix $\bar B_i = (b_j)_{j \in H_i}$ of $B$, such that the vertex $w_i$ is the unique
  solution of $\bar B_i y_i = \1$. The corresponding subsystem of the inequality system of $P \wedgeproduct Q$:
  \[
  \left(
    \arraycolsep2pt
    \begin{array}{ccccc}
      \bar A_0    &  \bar B_0   &     &       &      \\
      \bar A_1    &      & \bar B_1  &       &      \\
      \vdots &      &     &\ddots &      \\
      \bar A_{m-1}&      &     &       & \bar B_{m-1}  
    \end{array}
  \right)
  \begin{pmatrix}
    x \\ y_0 \\ y_1 \\ \vdots \\ y_{m-1}
  \end{pmatrix}
  \le
  \begin{pmatrix}
    \1 \\ \1 \\ \vdots \\ \1
  \end{pmatrix}  
  \]
  defines a face of the wedge product, since it is tight at the point $(\rowzero,w_0,\dots,w_{m-1})^t$ on the boundary of
  $P \wedgeproduct Q$.  
  Analogous to the proof of Proposition~\ref{prop:GenWedgePFace}, every point $x \in P$ corresponds to a unique point
  $(x,y_0,\dots,y_{m-1})$ on $F$ where $y_i$ is the unique solution of the subsystem $\bar B_i y_i = 1 - \bar A_i x $.
  So $F$ is affinely equivalent to~$P$. 
\end{proof}

Using either duality and~\cite[Cor.~2.4]{JoswigLutz2005}, or
Corollary~\ref{cor:SimpleGenWedges}, we obtain the following characterization of simple wedge products.

\begin{corollary}\label{cor:SimpleWedgeProducts}
  The wedge product $P \wedgeproduct Q$ of two polytopes $P$ and $Q$ is simple if and only if 
  \begin{compactitem}
  \item $P$ is a point and $Q$ is simple (trivial case) or 
  \item $P$ is simple and $Q$ is a simplex.
  \end{compactitem}
\end{corollary}

\section{The polyhedral surfaces $\Sigma_{p,2q}$}
\label{sec:polyhedral-surfaces}

In this section we study a particularly interesting polytope, the wedge product of a $p$-gon with a
$(q-1)$-simplex. We will identify a polyhedral surface in its boundary complex, and show that after suitable deformation
this surface can survive projections to $\R^4$ resp.\ $\R^3$ for certain parameters $p$ and $q$. This yields
realizations of the surfaces, the ``prisms'' over the surfaces, and of the dual surfaces.
 
\subsection{The wedge product of {$p$}-gon and {$(q-1)$}-simplex}
\label{sec:wedge-product-pq}

For $p\ge3$ and $q\ge1$,
the wedge product of a $p$-gon $C_p$ and a $(q-1)$-simplex $\simplex{q-1}$ will be denoted by
\[
\wp{p}{q-1}\ :=\ C_p \wedgeproduct \simplex{q-1}.
\]
This is a $(2 + p(q-1))$-dimensional polytope with $pq$ facets. By Corollary~\ref{cor:SimpleWedgeProducts}
it is simple.

Let us first fix some notation. We assume that the facets of $C_p$ are labeled in cyclic order, that is,
if $i,i'\in \Zp$ are indices of facets of $C_p$, then they intersect in a vertex of the $p$-gon if and only if $i'
\equiv i \pm 1 \mod p$. For $j \in \Zq$ we denote by $\overline{j}$ the set complement $\Zq \setminus j$ of $j$ in $\Zq$. A
vertex of the $(q-1)$-simplex $\simplex{q-1}$ is the intersection of any $q-1$ of its facets, hence for $j \in \Zq$ the
intersection $\bigcap_{j' \in \overline{j}} F_{j'}$ of the facets $F_{j'}$ ($j'\in\overline{j}$) of the simplex is a vertex. 
Theorem~\ref{prop:CombWedgeProducts} specializes as follows.
\begin{corollary}[Vertices of wedge product $\wp{p}{q-1}$]
  \label{cor:VerticesWP}
  Let $\wp{p}{q-1}$ be the wedge product of $p$-gon and $(q-1)$-simplex. Then the vertices of the wedge product
  $\wp{p}{q-1}$ correspond to the vectors $(H_0,\dots,H_{p-1})$ 
  with
  \begin{equation}
  (H_0,\dots,H_{p-1}) = 
  \begin{cases}
    (\overline{j_0},\dots,\overline{j_{i-1}},\Zq,\Zq,\overline{j_{i+2}},\dots,\overline{j_{p-1}}),&\text{\ or}\\
    (\Zq,\overline{j_1},\overline{j_2},\dots\dots\dots,\overline{j_{p-3}},\overline{j_{p-2}},\Zq)
  \end{cases}
  \end{equation}
  with $j_i \in [q]$ and $\overline{j_i} = \Zq \setminus\{j_i\}$. In other words, each vector that corresponds to a
  vertex has two cyclically adjacent $\Zq$ entries while all other entries are $q-1$ element subsets of~$\Zq$. The
  number of vertices is $p q^{p-2}$.
\end{corollary}

For the construction of the surface in the next section we are interested in the $p$-gon faces of $\wp{p}{q-1}$ that we
obtain from Proposition~\ref{prop:SpecialFacesWedgeProducts}.

\begin{corollary}[$p$-gon faces of the wedge product $\wp{p}{q-1}$]
  \label{cor:pGonsOfWP} 
  The faces of the wedge product $\wp{p}{q-1}$ of $p$-gon and $(q-1)$-simplex corresponding to the vectors
  \[
  (H_0,\dots,H_{p-1})\ =\ (\overline{j_0},\dots,\overline{j_{p-1}}),
  \] 
  where $j_k \in \Zq$ and $\overline{j_{k}} = \Zq \setminus j_k$ are $p$-gons. The number of such $p$-gons in
  $\wp{p}{q-1}$ is $q^p$.
\end{corollary}

\subsection{Combinatorial construction}
\label{sec:CombinatorialConstruction}

In this section we describe a combinatorially regular surface of type~$\{p,2q\}$, that is, a surface composed of $p$-gon
faces, whose vertices have uniform degree $2q$, and with a combinatorial automorphism group that acts transitively on
its flags. It will be a subcomplex of the $p$-gons of the wedge product~$\wp{p}{q-1} = C_p \wedgeproduct \simplex{q-1}$
of $p$-gon and $(q-1)$-simplex defined in the previous Section~\ref{sec:wedge-product-pq}.

To construct the surface we have to select certain of the $p$-gon faces of the wedge product. By
Corollary~\ref{cor:pGonsOfWP} we know that the $p$-gon faces of $\wp{p}{q-1}$ correspond to vectors
$(\overline{j_0},\ldots,\overline{j_{p-1}})$ with $\overline{j_{k}} = \Zq \setminus j_k$.

\begin{definition}[polytopal subcomplex $\WPsurf{p}{2q}$ of $\wp{p}{q-1}$]
\label{def:WedgeProductSurface}
For $p\ge3$ and $q\ge2$, 
the subcomplex $\WPsurf{p}{2q}$ is defined by the $p$-gon faces of the wedge product $\wp{p}{q-1}$ that correspond to
the following set of vectors:
\[
\WPsurf{p}{2q}\ =\ \Big\{(\,\overline{j_0},\ldots,\overline{j_{p-1}}\,)\ \Big|\ 
                             \sum_{k=0}^{p-1} j_k \equiv 0\textrm{ or }1 \mod q\Big\}.
\]
The subcomplex consists of all these $p$-gons, their edges and vertices. 
\end{definition}

Let us start with an easy observation on the faces contained in the subcomplex~$\WPsurf{p}{2q}$.

\begin{lemma}[Vertices and edges of~$\WPsurf{p}{2q}$]
  \label{lem:VertsAndEdgesOfWPsurf}
  The subcomplex~$\WPsurf{p}{2q}$ contains \emph{all} the vertices of~$\wp{p}{q-1}$.  It contains all the edges of type
  $(\,\overline{j_0},\dots,\overline{j_{k_0-1}},[q],\overline{j_{k_0+1}},\dots,\overline{j_{p-1}}\,)$.
  Thus the $f$-vector of~$\WPsurf{p}{2q}$ is given by
  \[
  (f_0,f_1,f_2)\ =\ (p ,p q, 2 q)q^{p-2}.
  \]   
\end{lemma}

In the following we will prove that the polytopal complex $\WPsurf{p}{2q}$ is a regular surface. We start by proving the
regularity.

\begin{proposition}[Regularity of the polytopal complex $\WPsurf{p}{2q}$]
\label{prop:RegularityOfWPsurf}
  The polytopal complex $\WPsurf{p}{2q}$ is regular, i.e. the combinatorial automorphism group acts transitively on its
  flags.
\end{proposition}
\begin{proof}
  We use four special combinatorial automorphisms of the subcomplex to show that the flag $\mathcal{F}_0:
  (\Zq,\Zq,\overline{0},\dots, \overline{0}) \subseteq (\Zq,\overline{0},\overline{0},\dots, \overline{0}) \subseteq
  (\overline{0},\overline{0},\overline{0},\dots, \overline{0})$ may be mapped onto any other flag.
  Acting on (index vectors of) vertices, they may be described as follows:
  \begin{align*}
    \mathsf{F}:\ &(\,\overline{j_0},\overline{j_1},\overline{j_2}, \dotsc, \overline{j_{p-1}}\,)\ \longmapsto\ 
    (\,\overline{j_{p-1}}, \dotsc , \overline{j_2},\overline{j_1},\overline{j_{0}}) &(\text{Flip})\\
    \mathsf{P}:\ &(\,\overline{j_0},\overline{j_1},\overline{j_2}, \dotsc, \overline{j_{p-1}}\,)\ \longmapsto\ 
    (\,\overline{q-j_{0}+1}, \overline{q-j_{1}}, \dotsc, \overline{q-j_{p-1}}\,) &(\text{Parity})\\
    \mathsf{R}:\ &(\,\overline{j_0},\overline{j_1},\overline{j_2},\ldots,\overline{j_{p-1}}\,)\ \longmapsto\ 
    (\,\overline{j_0+1},\overline{j_1-1},\overline{j_2}, \dotsc, \overline{j_{p-1}}\,)&(\text{Rotate})\\
    \mathsf{S}:\ &(\,\overline{j_0},\overline{j_1},\overline{j_2},\ldots,\overline{j_{p-1}}\,)\ \longmapsto\
    (\,\overline{j_{1}},\overline{j_2},\dotsc,\overline{j_{p-1}},\overline{j_0}\,)&(\text{Shift})
  \end{align*}
  All four maps act on the vectors representing the faces of the subcomplex $\WPsurf{p}{2q}$. 
  The map $\mathsf{P}$ changes the parity of the $p$-gon (that is, the value of $\sum j_k$ mod $2$),
  $\mathsf{S}$ shifts the vector cylically, $\mathsf{F}$ reverses the order of the vector, and $\mathsf{R}$ rotates
  around a vertex preserving parity. Hence by applying an appropriate combination of $\mathsf{F}$, $\mathsf{S}$, and
  $\mathsf{P}$ we may map an arbitrary flag to a flag of the type
  \[
  \mathcal{F} = 
  (\,\Zq,\Zq,\overline{j_2},\dots, \overline{j_{p-1}}\,) 
  \subseteq (\,\Zq,\overline{j_1},\overline{j_2},\dots, \overline{j_{p-1}}\,)
  \subseteq (\,\overline{j_{0}},\overline{j_{1}},\overline{j_{2}},\dots, \overline{j_{p-1}}\,).
  \]
  with $\sum j_i \equiv 0 \mod q$.  The two maps $\mathsf{S}$ and $\mathsf{R}$ do not change the parity of the $p$-gon.
  If we now apply the following sequence of $\mathsf{S}$ and $\mathsf{R}$ to the flag $\mathcal{F}_0$ we obtain the flag
  $\mathcal{F}$:
  \[
  \mathcal{F} = (\mathsf{S} (\mathsf{S} \mathsf{R}^{s_{p-2}}) (\mathsf{S} \mathsf{R}^{s_{p-3}}) \cdots (\mathsf{S}
  \mathsf{R}^{s_1}) (\mathsf{S} \mathsf{R}^{s_0}))(\mathcal{F}_0)
  \]
  where $s_{\l} = \sum_{k=0}^{\l} j_k$. Each of the $\mathsf{S}\mathsf{R}$ pairs adjusts one of the entries of the flag,
  and the entire sequence maps $\mathcal{F}_0$ to $\mathcal{F}$.
\end{proof}

\begin{remark}
  The symmetry group $\Aut{\WPsurf{p}{2q}}$ of the surface $\WPsurf{p}{2q}$ has order $4p q^{p-1}$.
  It is a quotient of the symmetry group $[p,2q]$
  of the universal (simply-connected) regular tiling of type $\{p,2q\}$, i.e.\ with $p$-gon faces and uniform vertex degree $2q$.
  (Depending on the parameters $p$ and $q$ these tilings are euclidean, spherical, or hyperbolic.) The group
  $\Aut{\WPsurf{p}{2q}}$ is generated by ``combinatorial reflections'' at the lines bounding a fundamental triangle of
  the barycentric subdivision of the surface. Our group $\Aut{\WPsurf{p}{2q}}$ is also a quotient of $G^{p,2q,r}$,
the quotient of $[p,2q]$ studied by Coxeter~\cite{Coxeter1939}, for suitable parameters~$r$.
\end{remark}

We are now able to prove the following result on the structure of our selected subcomplex.

\begin{theorem}[Properties of $\WPsurf{p}{2q}$]
  \label{thm:SurfaceInWedgeProduct}
  The subcomplex $\WPsurf{p}{2q}$ of the wedge product $\wp{p}{q-1} = C_p \wedgeproduct \simplex{q-1}$ of a $p$-gon and
  a $(q-1)$-simplex is a closed connected orientable regular $2$-manifold of type $\{p,2q\}$ with $f$-vector
  \[
  f(\WPsurf{p}{2q}) = (p, pq, 2q)q^{p-2}
  \]
  and genus $1+\tfrac{1}{2} q^{p-2}(pq - p - 2q)$.
\end{theorem}
\begin{proof}
  We start by proving that $\WPsurf{p}{2q}$ is a manifold, i.e. that at every vertex the $p$-gons form a $2$-ball.  By
  Proposition~\ref{prop:RegularityOfWPsurf} all the vertices are equivalent, so it suffices to consider the vertex $v =
  (\Zq,\Zq,\overline{0},\dots,\overline{0})$. The $p$-gons adjacent to the vertex $v$ correspond to the vectors
  $(\,\overline{j_0},\overline{j_1},\overline{0},\dots,\overline{0})$ with $j_0 + j_1 \equiv 0,1 \mod q$. Starting from
  the $p$-gon $(\overline{0},\overline{0},\overline{0},\dots,\overline{0})$ we obtain all the other $p$-gons adjacent to
  $v$ if we alternately increase the first component $j_0$ or decrease the second component $j_1$ as shown in
  Figure~\ref{fig:StarOfWPsurf}. The $2q$ edges joining the $p$-gons correspond to the vectors
  $(\Zq,\overline{j_1},\overline{0},\dots,\overline{0})$ or $(\,\overline{j_0},\Zq,\overline{0},\dots,\overline{0})$
  with $j_0, j_1 \in \Zq$. Thus the $p$-gons around each vertex form a $2$-ball and $\WPsurf{p}{2q}$ is a manifold with
  uniform vertex degree~$2q$.

  \begin{figure}[htb]
    \centering
    \psfrag{00}[lc]{\scriptsize{$(\overline{0},\overline{0})$}}
    \psfrag{01}[lc]{\scriptsize{$(\overline{0},\overline{1})$}}
    \psfrag{31}[lb]{\scriptsize{$(\overline{3},\overline{1})$}}
    \psfrag{32}[lb]{\scriptsize{$(\overline{3},\overline{2})$}}
    \psfrag{22}[cb]{\scriptsize{$(\overline{2},\overline{2})$}}
    \psfrag{23}[cb]{\scriptsize{$(\overline{2},\overline{3})$}}
    \psfrag{13}[cc]{\scriptsize{$(\overline{1},\overline{3})$}}
    \psfrag{10}[cc]{\scriptsize{$(\overline{1},\overline{0})$}}
    \includegraphics[width=4cm]{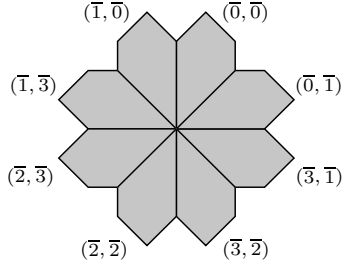}
    \caption{The $p$-gons (shaded areas) adjacent the vertex $(\Zq,\Zq,\overline{0},\overline{0},\overline{0})$ of $\WPsurf{p}{2q}$
      form a $2$-ball. In this case $p = 5$ and $q = 4$ and the pentagons are labeled by the first two entries of the
      vector representation.}
    \label{fig:StarOfWPsurf}
  \end{figure}

  We proceed by showing that the manifold is connected by constructing a sequence of $p$-gons connecting two arbitrary
  $p$-gons. Consider two $p$-gons $F = (\overline{j_0},\overline{j_1},\overline{j_2},\ldots,\overline{j_{p-1}})$ and $G =
  (\overline{j_0'},\overline{j_1'},\overline{j_2'},\ldots,\overline{j_{p-1}'})$. Then there exists a sequence of $p$-gons in the star of the
  vertex $v_0 = (\Zq,\Zq,\overline{j_2},\ldots,\overline{j_{p-1}})$ connecting $F$ to a $p$-gon $F_1 =
  (\overline{j_0'},\overline{j_1''},\overline{j_2},\ldots,\overline{j_{p-1}})$. For $k=1,\ldots,p-2$ we continue around the vertices $v_k =
  (\overline{j_1'},\ldots,\overline{j_{k-1}'},\Zq,\Zq,\overline{j_{k+2}},\ldots,\overline{j_{p-1}})$ to obtain
  $F_k=(\overline{j_1'},\ldots,\overline{j_{k}'},\overline{j_{k+1}''},\overline{j_{k+2}},\ldots,\overline{j_{p-1}})$ with the first $k$
  components equal to those of $G$. Then either $F_{p-1} = G$ or they share the common edge
  $(\overline{j_0'},\ldots,\overline{j_{p-2}'},\Zq$). 
  Hence we have shown so far that $\WPsurf{p}{2q}$ 
  is a closed connected equivelar $2$-manifold of type
  $\{p,2q\}$ without boundary.
  
  The surface $\WPsurf{p}{2q}$ consists of two families of $p$-gons $(\overline{j_0},\dots,\overline{j_{p-1}})$, 
  distinguished by ${j_0}+\dots+{j_{p-1}}\mod q$. We assign an orientation to the edges of the $p$-gons as follows
\[
  ([q],\overline{j_1},\overline{j_2},\dots,\overline{j_{p-1}})\ \rightarrow\ 
  (\overline{j_0},[q],\overline{j_2},\dots,\overline{j_{p-1}})\ \rightarrow\ 
  (\overline{j_0},\overline{j_1},[q],\dots,\overline{j_{p-1}})\ \rightarrow\ \dots  
\]
if ${j_0}+\dots+{j_{p-1}}\equiv0\mod q$, and
\[
  ([q],\overline{j_1},\overline{j_2},\dots,\overline{j_{p-1}})\ \leftarrow\ 
  (\overline{j_0},[q],\overline{j_2},\dots,\overline{j_{p-1}})\ \leftarrow\ 
  (\overline{j_0},\overline{j_1},[q],\dots,\overline{j_{p-1}})\ \leftarrow\ \dots  
\]
if ${j_0}+\dots+{j_{p-1}}\equiv1\mod q$.
In Figure~\ref{fig:OrientationWPsurf} this is illustrated for $p=5$. 
Since every edge is contained in one $p$-gon with sum~$\equiv0$ and one with sum~$\equiv1$ this yields
  a consistent orientation for the surface.

  As $\WPsurf{p}{2q}$ is an orientable manifold we calculate the genus of the surface from the $f$-vector given in
  Lemma~\ref{lem:VertsAndEdgesOfWPsurf} via the Euler characteristic:
  \begin{align*}
    g &= 1 - \tfrac{1}{2}\chi(\WPsurf{p}{2q}) = 1 + \tfrac{1}{2}((q-1)p - 2q)q^{p-2}.
  \end{align*}
\vskip-5.5mm

\end{proof}

  \begin{figure}[t]
    \begin{minipage}{0.35\textwidth}
      \centering
      \psfrag{0}[rb]{\tiny{$(\Zq,\overline{j_1},\overline{j_2},\overline{j_3},\overline{j_4})$}}
      \psfrag{1}[lb]{\tiny{$(\overline{j_0},\Zq,\overline{j_2},\overline{j_3},\overline{j_4})$}}
      \psfrag{2}[lb]{\tiny{$(\overline{j_0},\overline{j_1},\Zq,\overline{j_3},\overline{j_4})$}}
      \psfrag{3}[ct]{\tiny{$(\overline{j_0},\overline{j_1},\overline{j_2},\Zq,\overline{j_4})$}}
      \psfrag{4}[rb]{\tiny{$(\overline{j_0},\overline{j_1},\overline{j_2},\overline{j_3},\Zq)$}}

      \psfrag{01}{}
      \psfrag{12}{}
      \psfrag{23}{}
      \psfrag{34}{}
      \psfrag{40}{}

      \psfrag{pg}[cb]{\tiny$\;\;(\overline{j_0},\overline{j_1},\overline{j_2},\overline{j_3},\overline{j_4})$}
      \psfrag{or}[cc]{$\sum j_k \equiv \mathbf{0}$}

      \includegraphics[height=4cm]{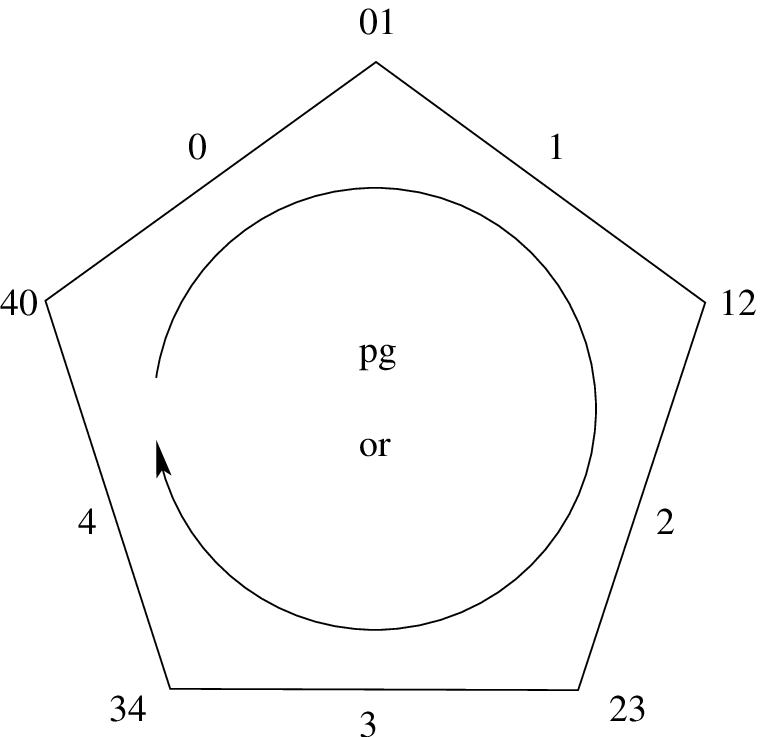}
    \end{minipage}
    \hfill
    \begin{minipage}{0.35\textwidth}
      \centering
      \psfrag{0}[rb]{\tiny{$(\Zq,\overline{j_1},\overline{j_2},\overline{j_3},\overline{j_4})$}}
      \psfrag{1}[lb]{\tiny{$(\overline{j_0},\Zq,\overline{j_2},\overline{j_3},\overline{j_4})$}}
      \psfrag{2}[lb]{\tiny{$(\overline{j_0},\overline{j_1},\Zq,\overline{j_3},\overline{j_4})$}}
      \psfrag{3}[ct]{\tiny{$(\overline{j_0},\overline{j_1},\overline{j_2},\Zq,\overline{j_4})$}}
      \psfrag{4}[rb]{\tiny{$(\overline{j_0},\overline{j_1},\overline{j_2},\overline{j_3},\Zq)$}}

      \psfrag{01}{}
      \psfrag{12}{}
      \psfrag{23}{}
      \psfrag{34}{}
      \psfrag{40}{}

      \psfrag{pg}[cb]{\tiny$\;\;(\overline{j_0},\overline{j_1},\overline{j_2},\overline{j_3},\overline{j_4})$}
      \psfrag{or}[cc]{$\sum j_k \equiv \mathbf{1}$}

      \includegraphics[height=4cm]{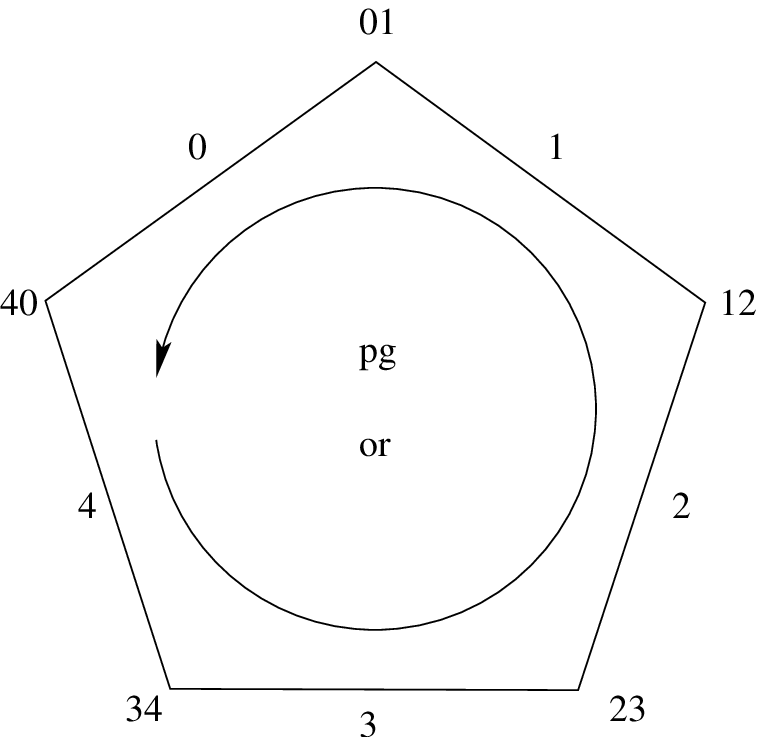}
    \end{minipage}
    \caption{The orientation of the $p$-gons in the surface $\WPsurf{p}{2q}$, for $p=5$.}
    \label{fig:OrientationWPsurf}
  \end{figure}

For $p=3$ we obtain a regular surface $\WPsurf{3}{2q}$ of type $\{3,2q\}$ with $f$-vector $(3q,3q^2,2q^2)$ in the wedge
product $\wp{3}{q-1}$. The genus of the surface is $1+ \tfrac{1}{2}q(q-3)$ and thus quadratic in the number of vertices.
Unfortunately, the wedge product of a triangle and a $(q-1)$-simplex is a polytope of dimension $3q-1$ with $3q$ facets,
hence a $(3q-1)$-simplex. So our construction does not provide an ``interesting'' realization of the surface. The surface
$\WPsurf{3}{2q}$ is well known and occurs already in Coxeter and Moser~\cite{CoxeterMoser1980}. For $q=2$ the surface is
the octahedron and for $q=3$ Dyck's Regular Map. For Dyck's regular map there exist two realizations in $\R^3$, one by
Bokowski~\cite{Bokowski1989} and a more symmetric one by Brehm~\cite{Brehm1987}. 

For $q=2$ the surface $\WPsurf{p}{4}$ is the surface of type~$\{p,4\}$ constructed by McMullen,
Schulz, and Wills~\cite[Sect. 4]{McMullen1983}. In their paper they construct a realization of the surface 
directly in $\R^3$. Their construction also provides two additional parameters~$m$ and~$n$. 
Our construction produces the McMullen--Schulz--Wills surface  with parameters $m=2$ and $n=2$.

So our surface generalizes two interesting families of surfaces. As we will see, for some parameters it also provides a
new way of realizing the surface in the boundary complex of a $4$-polytope and by orthogonal projection in~$\R^3$.

In contrast to simplicial complexes, which may always be realized in a high-dimensional simplex, there is no fool-proof
strategy for realizing general polytopal complexes. For abstract non-simplicial polyhedral $2$-manifolds
in general not even a realization in $\R^N$ for high $N$ is possible. For example, equivelar surfaces of type $\{p,2q+1\}$ are
not realizable in general:

\begin{proposition}[Betke and Gritzmann \cite{BetkeGritzmann1982}]
  Let $S \subset \R^N$ be an equivelar polyhedral $2$-manifold of type $\{p,2q+1\}$ with $q \geq 1$ in $\R^d$.  Then
  $2(2q+1) \geq p+1$.
\end{proposition}

If realized in some $\R^N$, a polyhedral surface can be embedded into $\R^5$ via an arbitrary general position
projection.  Combining this observation with Theorem~\ref{thm:SurfaceInWedgeProduct} we get the following corollary.

\begin{corollary}
  \label{cor:SurfaceInR5}
  The surfaces $\WPsurf{p}{2q}$ of Theorem~\ref{thm:SurfaceInWedgeProduct} ($p\ge3$, $q\ge2$) can be realized in $\R^5$.
\end{corollary}

\section{Realizing the surfaces $\WPsurf{p}{4}$ and $\WPsurf{p}{4}^*$.}\label{sec:realizing}

In the following we provide a construction for the surfaces $\WPsurf{p}{4}$ and $\WPsurf{p}{4}^*$ in $\R^3$ via
projection. We construct a realization of $\wp{p}{1}$ that allows for a projection of the surface into the boundary of a
$4$-dimensional polytope. A particular property of the embedding will be that all the faces of the surface lie on the
``lower hull'' of the polytope. In this way we obtain the surface by an orthogonal projection to $\R^3$ and we do not
need to take the Schlegel diagram. The dual surface $\WPsurf{p}{4}^*$ is constructed via a projection of the product
$\wp{p}{1} \times I$ of the wedge product with an interval~$I$.

\subsection{Projections and linear algebra}
\label{sec:ProjectionsAndLA}

Under a projection $\pi$ the $k$-faces of a polytope $P$ may be mapped to $k$-faces of $\pi(P)$, to lower dimensional
faces of $\pi(P)$, to subsets of faces of $\pi(P)$, or into the interior of $\pi(P)$. We restrict ourselves to the
nicest case of preserved faces (formerly known as ``strictly preserved faces''~\cite{Ziegler2004a}), as
given by the following definition.

\begin{definition}[preserved faces] Let $P \subset \R^d$ be a polytope and $Q = \pi(P)$ be the image of $P$ under the
  affine projection map $\pi:\R^d \rightarrow \R^k$. A non-empty face $F$ of $P$ is \emph{preserved} by $\pi$ if
    \begin{compactenum}[\rm (D1)]
        \item $\pi(F)$ is a face of $Q$,
        \item $\pi(F)$ is combinatorially equivalent to~$F$, and
        \item the preimage $\pi^{-1}(\pi(F)) \cap P$ is $F$.
    \end{compactenum}
\end{definition}

These preserved faces are stable under perturbation: If a face $F$ of a polytope $P$ is preserved under projection, then
it is also preserved by any small perturbation of the projection. Moreover, in the case of a simple polytope $P$ we
may consider small perturbations of (the facet-defining inequalities/hyperplanes of) the polytope $P$ that do not
change the combinatorial type, and for sufficiently small such perturbations the 
corresponding face $\tilde{F}$ of the slightly perturbed polytope $\tilde{P}$ is also preserved under the same
projection. To be able to apply this definition to our polytopes we use the following lemma which connects the preserved
faces to the inequality description of a polytope.

\begin{lemma}[Preserved faces: linear algebra version~\cite{Ziegler2004a}]
\label{lem:projection}
Let $P \subset \R^d$ be a $d$-polytope with facets $a_ix \leq 1$ for $i\in [m]$, $F$ a non-empty face of $P$, and $H_F$
the index set of the inequalities that are tight at~$F$. Then $F$ is preserved by the projection to the first $k$
coordinates if and only if the facet normals truncated to the last $d-k$ coordinates 
$\{a_i^{(d-k)}: i\in H_F\}$ positively span~$\R^{d-k}$.
\end{lemma}

The above Lemma~\ref{lem:projection} suffices to obtain a projection of the wedge product to~$\R^4$ that preserves the
surfaces for some parameters. To get a realization of the surface we might construct the Schlegel diagram of the projected
polytope in $\R^3$ via a central projection. To avoid this central projection  
we will project the surface onto the lower hull of the $4$-polytope.

\begin{definition}[lower hull]\label{def:LowerHull}
  The \emph{lower hull} of a polytope with respect to some coordinate direction $x_\ell$ is the polytopal complex
  consisting of all faces that have a normal vector with negative $x_\ell$-coordinate.
\end{definition}

So the following is the lemma we will use to prove the realization of the surface in the lower hull of a projected wedge
product.

\begin{lemma}[Preserved faces on the lower hull] 
  \label{lem:lowerHull}
  Let $P = \{ x \in \R^d\ |\ a_ix \leq 1, i\in [m]\} \subset\R^d$ be a polytope, $F$ a non-empty face of $P$, and $H_F
  \subset [m]$ the index set of the inequalities that are tight at~$F$. Then $\pi(F)$ is on the lower hull with respect
  to $x_{k-1}$ of the projection to the first $k$ coordinates if and only if
  \begin{compactenum}[\rm (L1)]
  \item \label{item:positivelySpan} the facet normals truncated to the last $d-k$ coordinates $a_i^{(d-k)}$ with $i \in
    H_F$ positively span $\R^{d-k}$, and
  \item \label{item:lowerHull} there exist $\lambda_i \geq 0$ such that $\nu = \sum_{i\in H_F} \lambda_i a_i$ with
    $(\nu_k,\dots,\nu_{d-1}) = 0$ and $\nu_{k-1} < 0$.
  \end{compactenum}
\end{lemma}
\begin{proof}
  The first part~(L\ref{item:positivelySpan}) of this lemma is exactly the Projection Lemma~\ref{lem:projection} and the
  second part~(L\ref{item:lowerHull}) corresponds to Definition~\ref{def:LowerHull}.
\end{proof}

\subsection{Projection of the surface to $\R^4$ and to $\R^3$}
\label{sec:ProjectionToR4AndR3}

We are now ready to state our main result about the projections of the surfaces contained in the wedge products of
$p$-gons and intervals.

\begin{theorem}
  \label{thm:pdwp}
  The wedge product $\wp{p}{1}= C_p \wedgeproduct \simplex{1}$ of dimension~$2+p$ has a realization in $\R^{2+p}$ such
  that all the faces corresponding to the surface $\WPsurf{p}{4}\subset \wp{p}{1}$ are preserved by the
  projection to the last four/three coordinates.

  This realizes $\WPsurf{p}{4}$ as a subcomplex of a polytope boundary in~$\R^4$,
  and as an embedded polyhedral surface in~$\R^3$.
\end{theorem}

\begin{proof}
  We proceed in two steps. In the first step we construct a wedge product of a $p$-gon with a $1$-simplex and
  describe a suitable deformation.  In the second step we use the Projection Lemma~\ref{lem:lowerHull} to show that the
  projection of the deformed wedge product to the first four coordinates preserves all the $p$-gons of the surface
  $\WPsurf{p}{4}$. Furthermore, all the faces of the projected surface lie on the lower hull of the projected polytope
  and hence the surface may be realized by an orthogonal projection to the last four/three coordinates.

  Let the $p$-gon $C_p$ be given by $C_p = \{x \in \R^2\ |\ a_i x \leq 1, i \in \Zp\}$ with facets in cyclic order, let
  $\eps > 0$ be a small positive number and for a $1$-simplex take the set $\simplex{1} = \{y\in \R\ |\ \pm\eps y \leq 1\}$.
  Then by Definition~\ref{def:WedgeProduct} the inequality description of the wedge product $\wp{p}{1}$ is
  \begin{align*}
    \left(
    \begin{array}{ccc|cccc}
      a_0    &\pm\eps   &         &       &         &         &       \\
      a_1    &          & \pm\eps &       &         &         &       \\
      a_2    &          &         &\pm\eps&         &         &       \\
      \vdots &          &         &       & \ddots  &         &       \\ 
      a_{p-2}&          &         &       &         & \pm\eps &       \\
      a_{p-1}&          &         &       &         &         &\pm\eps\\
    \end{array}
    \right)
    \begin{pmatrix}
      x \\ y_0 \\ y_1 \\ \vdots \\ y_{p-1}
      \end{pmatrix}
      \leq 
      \begin{pmatrix}
        1 \\ 1 \\ 1 \\ \vdots \\ 1
      \end{pmatrix}.
  \end{align*}
  Each of the rows in the matrix $C$ corresponds to two facets -- one for each sign.  
  Since $\wp{p}{1}$ is a simple polytope we may perturb the facet normals of the wedge product in
  the following way without changing the combinatorial structure.
  Thus for $M>0$ large enough we obtain a realization of $\wp{p}{1}$ of the form
  \[
  \left(
    \begin{array}{ccc|cccc}
    a_0      &\pm\eps           & -\frac{1}{M}       &-\frac{1}{M^2} & \cdots          & -\frac{1}{M^{p-2}} &-\frac{1}{M^{p-1}}\\
    a_1      &                  & \pm\eps            & \frac{1}{M}   &                 &                    &                 \\
    a_2      &                  &                    & \pm\eps       & \frac{1}{M}     &                    &                 \\
    \vdots   &                  &                    &               &\ddots           & \ddots             &                 \\
    a_{p-2}  &                  &                    &               &                 &\pm\eps             &\frac{1}{M}      \\
    a_{p-1}  &                  &                    &               &                 &                    & \pm\eps         \\
  \end{array}
  \right)
  \begin{pmatrix}
     x \\ y_0 \\y_1\\ \vdots \\ y_{p-1}
  \end{pmatrix}
  \leq 
  \begin{pmatrix}
    1\\ 1 \\ 1 \\ \vdots \\ 1
  \end{pmatrix}.
  \]
  We rescale the inequalities of the wedge product and replace the variables by multiplying the $i$-th pair of rows with
  $M^{p-1-i}$ and setting $y_i' = M^{p-1-i}y_i$ to get
  \[
  \left(
    \begin{array}{ccc|cccc}
    M^{p-1}a_0      &\pm\eps           & -1                 &-1             & \cdots          &        -1          &      -1 \\
    M^{p-2}a_1      &                  & \pm\eps            &       1       &                 &                    &         \\
    M^{p-3}a_2      &                  &                    & \pm\eps       &       1         &                    &         \\
           \vdots   &                  &                    &               &\ddots           & \ddots             &         \\
    M      a_{p-2}  &                  &                    &               &                 &\pm\eps             &      1  \\
           a_{p-1}  &                  &                    &               &                 &                    & \pm\eps \\
  \end{array}
  \right)
  \begin{pmatrix}
    x \\ y_0' \\y_1'\\ \vdots \\ y_{p-1}'
  \end{pmatrix}
  \leq 
  \begin{pmatrix}
    M^{p-1}\\ M^{p-2} \\ M^{p-3} \\ \vdots \\ 1
  \end{pmatrix}.
  \]
  The above modifications do not change the combinatorial structure: scaling the inequalities changes nothing, and the
  change of variables is just a scaling of the coordinate axes. According to Definition~\ref{def:WedgeProductSurface},
  the surface $\WPsurf{p}{4}$ contains the following
  $p$-gon faces of $\wp{p}{1}$:
  \[
  \WPsurf{p}{4} = \Big\{(\overline{j_0},\ldots,\overline{j_{p-1}}) \in [2]^p\ \big|\ \sum_{k=0}^{p-1} j_k \equiv 0,1 \mod 2\Big\}.
  \]
  Since $[2] = \{0,1\}$, the surface $\WPsurf{p}{4}$ contains \emph{all} the ``special''
  $p$-gons of $\wp{p}{1}$ specified by Corollary~\ref{cor:pGonsOfWP}. Each of the $p$-gons is
  obtained by intersecting~$p$ facets with one facet chosen from each pair of rows, i.e.\ the $p$-gon
  $(\overline{j_0},\ldots,\overline{j_{p-1}})$ corresponds to a choice of signs $((-1)^{j_0},\ldots,(-1)^{j_{p-1}})$ in the 
  above matrix. So the normals to the facets containing the $p$-gon $(\overline{j_0},\ldots,\overline{j_{p-1}})$ are:

  \begin{align}
    \left(
      \begin{array}{ccc|cccc}
    M^{p-1}a_0      &(-1)^{j_0}   \eps & -1                 &-1             & \cdots  &        -1          &      -1 \\
    M^{p-2}a_1      &                  & (-1)^{j_1}   \eps  &       1       &         &                    &         \\
    M^{p-3}a_2      &                  &                    &(-1)^{j_3}\eps &       1 &                    &         \\
           \vdots   &                  &                    &               &\ddots   & \ddots             &         \\
    M      a_{p-2}  &                  &                    &               &         &(-1)^{j_{p-2}} \eps &      1  \\
           a_{p-1}  &                  &                    &               &         &                    &(-1)^{j_{p-1}}\eps 
      \end{array}
    \right)\label{eq:pgon_normals}
  \end{align}
  By Lemma~~\ref{lem:projection} a $p$-gon survives the projection to the first four coordinates 
  if rows of the right-hand part of the matrix (formed by the last~\mbox{$p-2$} columns)
  are positively spanning. Since $\eps$ is very small and the conditions of the
  projection lemma are stable under perturbation, the last $p-2$ columns are positively spanning independent of the
  choice of signs $(-1)^{j_i}$. Consequently all the $p$-gons survive the projection to the first four coordinates.

  This deformed realization of the wedge product has the additional property that all the $p$-gon faces of the surface
  have a face normal that has a negative fourth ($y_1$) coordinate as required in Lemma~\ref{lem:lowerHull}: The normal
  cone of a $p$-gon face is spanned by the normals of the facets containing the $p$-gon given by the matrix
  in~\eqref{eq:pgon_normals}. Since the $-1$ in the $y_1$ coordinate of the first row dominates the $y_1$ coordinates of
  the other normals, the normal cone of the projected $p$-gon contains a vector $\nu = (\nu_x,\nu_y)\in
  \R^{2+p}$ with $\nu_{y_j} = 0$ for $j=2,\dots,p-1$ and negative $\nu_{y_1} < 0$. Hence the $p$-gons of the
  surface lie on the lower hull of the projected polytope.
  
  Thus we get a coordinatization of the surface by orthogonal projection to the first three coordinates: 
  There is no need for a Schlegel projection.
\end{proof}

\subsection{Surface duality and polytope duality}
\label{sec:SurfaceAndPolytopeDuality}

In Section~\ref{sec:ProjectionToR4AndR3}, we have obtained a realization of the surface $\WPsurf{p}{4}$ 
as a subcomplex of a polytope boundary in~$\R^4$, and thus as an embedded polyhedral surface in~$\R^3$.
Now our ambition is to derive from this a realization of the dual surface $\WPsurf{p}{4}^*$.
This is not automatic: For this the dimension~$4$ is crucial, and also we need
that not only the surface $\WPsurf{p}{4}$, but also the ``prism'' $\WPsurf{p}{4}\times I$
over the surface embeds into a $4$-polytope as a subcomplex.

(Indeed, a surface embedded as a subcomplex in the boundary of a $d$-polytope $P$
exhibits a collection of faces of dimensions $0$, $1$, and $2$.
This corresponds to faces of dimensions $d-1$, $d-2$ and $d-3$ in the boundary
of the dual polytope $P^*$. These do not form a subcomplex unless $d=3$;
for larger $d$ this is a collection of high-dimensional faces that just have the
inclusion relations dictated by the face poset of $\WPsurf{p}{4}^*$.)

For the following, the \emph{prism} over a cell complex or polyhedral complex $\Sigma$ refers to
the product $\Sigma\times I$ with an interval $I=[0,1]$, equipped with the obvious
cellular structure that comes from the cell decompositions of $\Sigma$ (as given) and
of $I$ (with two vertices and one edge).
In particular, if $\Sigma$ is a polytope (with the canonical face structure), then
then $\Sigma\times I$ is the prism over $\Sigma$ in the classical sense of polytope theory.

\begin{theorem}
  \label{thm:DualSurface}
  The prism $\wp{p}{1} \times I$ over the wedge product $\wp{p}{1}$ 
  has a realization such that the prism over the surface $\WPsurf{p}{4} \times I$ survives the
  projection to $\R^4$ resp.~$\R^3$. 

  Furthermore, the dual of the projected $4$-polytope contains the dual surface $\WPsurf{p}{4}^*$ as a subcomplex, thus 
  by constructing a Schlegel diagram we obtain a realization of the surface $\WPsurf{p}{4}^*$ in~$\R^3$.
\end{theorem}

\begin{proof}
  The proof follows the same line as the proof of Theorem~\ref{thm:pdwp}. For $0 < \delta \ll 1$ we construct the
  product $\wp{p}{1} \times I$ of an interval $\smallSetOf{z \in \R}{\pm\delta z \leq \1}$ with the orthogonal wedge product which has the
  following inequality description:

  \begin{align*}
    \left(
    \begin{array}{ccccccc|c}
      a_0    &\pm\eps   &         &       &         &         &        &           \\
      a_1    &          & \pm\eps &       &         &         &        &           \\
      a_2    &          &         &\pm\eps&         &         &        &           \\
      \vdots &          &         &       & \ddots  &         &        &           \\ 
      a_{p-2}&          &         &       &         & \pm\eps &        &           \\
      a_{p-1}&          &         &       &         &         &\pm\eps &           \\ \hline
             &          &         &       &         &         &        & \pm\delta
    \end{array}
    \right)
    \begin{pmatrix}
        x \\ y_0 \\y_1\\ \vdots \\ y_{p-1} \\ z
      \end{pmatrix}
      \leq 
      \begin{pmatrix}
        1 \\ 1 \\ 1\\ \vdots \\ 1 \\ 1 \\ 1
      \end{pmatrix}.
  \end{align*}
  As in Theorem~\ref{thm:pdwp} we will project onto the first $4$ coordinates indicated by the vertical line in the next
  matrix. We perform a suitable deformation and obtain a deformed polytope combinatorially equivalent to~$\wp{p}{1}
  \times I$:
  \begin{align*}
    \left(
    \begin{array}{ccc|ccccc}
      M^{p  }a_0    &\pm\eps   &   -1    &  -1   &   -1    &   -1    &  -1    &  -1       \\
      M^{p-1}a_1    &          & \pm\eps &   1   &         &         &        &           \\
      M^{p-2}a_2    &          &         &\pm\eps&    1    &         &        &           \\
             \vdots &          &         &       & \ddots  & \ddots  &        &           \\ 
      M^{ 2 }a_{p-2}&          &         &       &         & \pm\eps &   1    &           \\
      M^{ 1 }a_{p-1}&          &         &       &         &         &\pm\eps &   1       \\ 
        \rowzero    &          &         &       &         &         &        & \pm\delta
    \end{array}
    \right)
    \begin{pmatrix}
      x \\ y_0' \\y_1'\\ \vdots \\ y_{p-1}' \\ z
    \end{pmatrix}
      \leq 
      \begin{pmatrix}
        M^{p} \\ M^{p-1} \\ M^{p-2} \\ \vdots \\ M^{2} \\ M^1 \\ 1
      \end{pmatrix}.
  \end{align*}
  The matrix has the same structure as the one used in Theorem~\ref{thm:pdwp} except for the~$0$ in the last row of the
  first column. The prism over the surface $\WPsurf{p}{4} \times I$ is a union of prisms over $p$-gons, where each
  prism is identified with the corresponding vector $(\overline{j_0},\ldots,\overline{j_{p-1}})$ of the $p$-gon face in the
  surface $\WPsurf{p}{4}$. The normals of the facets containing a prescribed $p$-gon prism are:
  \begin{align*}
    \left(
   \begin{array}{ccc|ccccc}
      M^{p  }a_0    &(-1)^{j_0}\eps  &   -1           &  -1          &   -1    &   -1             &  -1             &-1\\
      M^{p-1}a_1    &                &(-1)^{j_1}\eps  &   1          &         &                  &                 &  \\
      M^{p-2}a_2    &                &                &(-1)^{j_2}\eps&    1    &                  &                 &  \\
             \vdots &                &                &              & \ddots  & \ddots           &                 &  \\ 
      M^{ 2 }a_{p-2}&                &                &              &         &(-1)^{j_{p-2}}\eps&   1             &  \\
      M^{ 1 }a_{p-1}&                &                &              &         &                  &(-1)^{j_{p-1}}\eps& 1
    \end{array}
    \right)
  \end{align*}
  As in the previous proof, the rows of the right-hand part of the matrix are positively spanning because they are positively
  spanning for $\eps = 0$ and the given configuration is only a perturbation since $\eps$ is very small. Further the
  $-1$ in the $y_1$ coordinate of the first row dominates and yields a normal with negative $y_1$ coordinate for
  the prisms over the $p$-gons. So the prism over the surface survives the projection to a $4$-dimensional polytope and
  lies on its lower hull using Lemma~\ref{lem:lowerHull}. This way we obtain a realization of the prism over the surface
  in $\R^3$ by orthogonal projection.

  Looking at the face lattice of the projected polytope we observe that it contains three copies of the face lattice of
  the surface $\WPsurf{p}{4}$ -- the top and the bottom copy and another copy raised by one dimension corresponding to
  the prism faces connecting top and bottom copy shown in Figure~\ref{fig:face_lattice_prism} (left). The face lattice of
  the dual polytope contains three copies of the face lattice of the dual surface. One of those copies based at the
  vertices corresponds to the dual surface contained in the $2$-skeleton of the dual polytope.
\end{proof}

  \begin{figure}[tbh]
    \centering
    \begin{minipage}{.4\textwidth}
      \includegraphics[width=\linewidth]{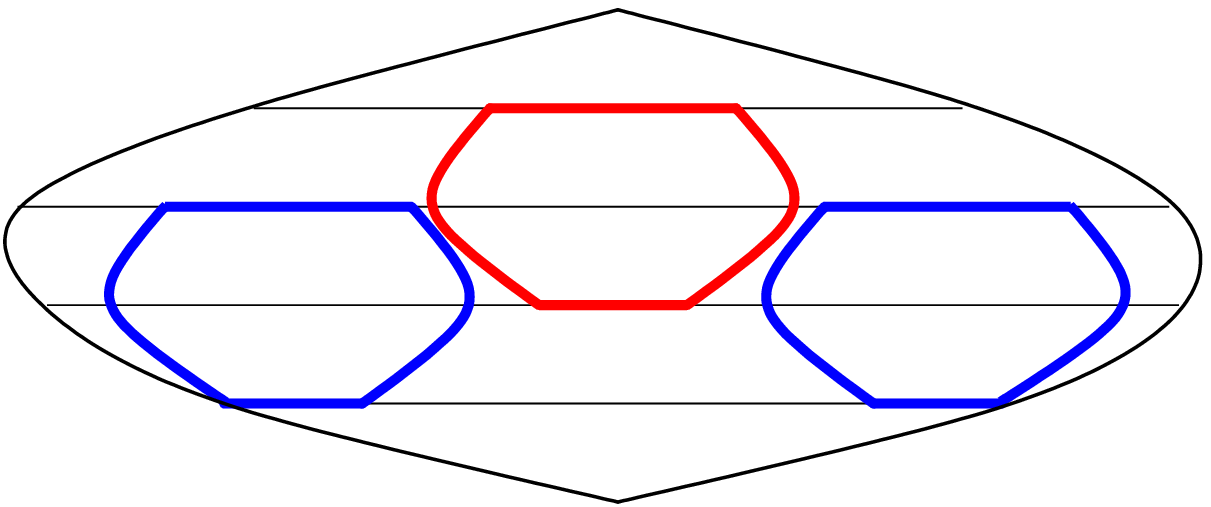}
    \end{minipage}
    \hspace{1cm}
    \begin{minipage}{.4\textwidth}
      \includegraphics[width=\linewidth, angle = 180]{face_lattice_prism}
    \end{minipage}
    \caption{The face lattice of a $4$-polytope containing the face poset of the prism over the surface $\WPsurf{p}{4}$
      (left). The face lattice of the dual polytope contains the face poset of dual surface (right).
      \label{fig:face_lattice_prism}}
  \end{figure}

\subsection{Moduli of the projected surfaces}
\label{sec:Moduli}

We introduce a new technique to estimate the number of moduli of surfaces that arise from projections of high
dimensional simple polytopes. Given a surface~$S$ realized in $\R^3$ with planar convex polygons
(or a more general polytopal complex, e.g.\ a convex polytope in some $\R^d$), the \emph{number of
  moduli} $\moduli{S}$ is the local dimension of the realization space in the neighborhood of the realization~$S$. We focus on
\emph{non-trivial} moduli of a surface, that is, deformations that are not due to a projective transformation.

A way to estimate the number of moduli for a polyhedral surface~$S \subset \R^3$ is to count the ``degrees of freedom''
and subtract the number of ``constraints''. This is captured in a meta-theorem by
Crapo~\cite{Crapo1985:CombTheoryStructures} on much more general incidence configurations,
saying that the number of moduli of a configuration is ``\emph{the number of
degrees of freedom of the vertices minus the number of generic constraints plus the number of hidden (incidence)
theorems}.''  

For a polyhedral surface $S \subset \R^3$,  
the $f_0(S)$ vertices together have~$3 f_0(S)$ degrees of freedom.
There are $p-3$ constraints needed for a $p$-gon to guarantee that its vertices lie in a plane.
Thus we obtain  $f_{02}(S)-3f_2(S)$ constraints, where $f_{02}(S)$
counts the number of vertex-polygon incidences. Finally we subtract $15$ degrees of freedom corresponding to the choice
of a projective basis. Thus the naive ``degrees of freedom minus number of constraints'' count yields the following
estimate for the number of moduli of a realization~$S \subset \R^3$
\begin{align*}
\moduli{S}\ &\ge 3 f_0(S) - (f_{02}(S) - 3 f_2(S))-15 
\\
&\qquad = 3 f_0(S) - 2 f_1(S) + 3f_2(S)-15,
\end{align*}
For the surfaces~$\WPsurf{p}{4}$ of Section~\ref{sec:ProjectionToR4AndR3} this amounts to 
\[
\moduli{\WPsurf{p}{4}} \ge 2^{p-2}(3p-4p+12)-15 = 2^{p-2}(12-p)-15.
\]
This estimate becomes negative and hence useless for $p\ge12$. It shows, however, that there is a huge number
(growing exponentially with $p$) of hidden incidence theorems in the polyhedral surfaces $\WPsurf{p}{4}$.
On the other hand, we will demonstrate here that the surfaces also have a large number of moduli.

Since our surfaces are constructed using projections of high-dimensional simple polytopes, we may use the moduli of the
simple wedge products to obtain moduli for the surfaces. The moduli of a simple $d$-polytope~$P$ are easily described in
terms of their facets: Every facet inequality of~$P$ may be slightly perturbed without changing the combinatorial type.
This yields~$d$ moduli for every facet, i.e.\ a total of~$d {\cdot}f_{d-1}(P)$ moduli. This parametrization of the
realization space of a simple polytope is not well suited to our projection purposes. We introduce a new parametrization
in terms of a subset of the vertices, which allows us to understand the moduli under projections. We propose the
following definition.

\begin{definition}[affine support set] 
  \label{dfn:AffineSupportSet}
  A subset~$A$ of the vertex set of a simple polytope is an \emph{affine support set} if in every realization of the polytope
  and for every facet~$F$ the restriction $A \cap F$ of the subset to the facet is affinely independent.
\end{definition}

As a first observation we obtain a lower bound on the dimension of the realization space of a polytope from the size of
an affine support set.

\begin{lemma}
  Let $P \subset \R^d$ be a simple polytope with vertex set~$V$ and~$A \subseteq V$ an affine support set of~$P$. Then the
  number of moduli $\moduli{P}$
  (i.e.\ the local dimension of the realization space of~$P$) is bounded from below by $d {\cdot}|A|$.
\end{lemma}

\begin{proof} 
  Any small  perturbation
  of the vertices in~$A$ can be extended to another realization of~$P$. Hence we obtain $d {\cdot}|A|$ independent moduli.
\end{proof}

Observe that the cardinality of the affine support set~$A$ is bounded by the number of facets of the simple polytope~$P$,
since an affinely independent set in every facet contains at most~$d$ vertices and every vertex is contained in~$d$
facets. (This also follows from the observation that the realization space of a simple polytope has dimension $d {\cdot}
f_{d-1}(P)$.) Not every simple polytope has an affine support set~$A$ of cardinality~$f_{d-1}(P)$, as exemplified by
the triangular prism:
The prism has~$5$ facets, but every choice of~$5$ vertices contains the four vertices of a quadrilateral facet, which are
not affinely independent.
We use the above lemma to obtain lower bounds on the dimensions of the realization spaces of projected polytopes.

\begin{corollary}
  \label{cor:ModuliProjection}
  Let $P \subset \R^d$ be a simple polytope with vertex set~$V$, let $A \subseteq V$ be an affine support set and~$\pi:\R^d \to
  \R^e$ a projection that preserves all the vertices in~$A$. Then the number of moduli of~$\pi(P)$ is
  at least $\moduli{\pi(P)}\ge e\cdot|A|$. This bound also holds for arbitrary
  subcomplexes of $\pi(P)$ that contain all the vertices of~$A$.
\end{corollary}

So to obtain a good lower bound on the number of moduli of the surfaces~$\WPsurf{p}{4}$ in~$\R^3$, we need to determine a large
affine support set for the corresponding wedge product~$\wp{p}{1}$.

\begin{theorem}
  Let~$\wp{p}{1}$ be the wedge product of $p$-gon and $1$-simplex for $p \ge 3$. Consider the following vertices for
  $k=0,\dotsc,p-2$:
  \begin{align*}
    v_k &= (\underbrace{0\ 0\ \dots\ 0}_{k}\ [2]\ [2]\ \underbrace{1\ 1\ \dots\ 1}_{p-k-2}) 
    &     v_{p-1} & = ( [2]\ 0\ \dots\ 0\ [2]) \\
    \bar{v}_k &= (\underbrace{1\ 1\ \dots\ 1}_{k}\ [2]\ [2]\ \underbrace{0\ 0\ \dots\ 0}_{p-k-2}) &
    \bar{v}_{p-1} & = ( [2]\ 1\ \dots\ 1\ [2]) 
  \end{align*}
  Then $A = \{ v_k \,|\, k = 0,\dotsc,p-1\} \cup \{ \bar{v}_k \,|\, k = 0,\dotsc,p-1\}$ is an affine support set of
  cardinality $2p$.
\end{theorem}
\begin{proof}
  Using the symmetry of the wedge product and the chosen subset $A$ it suffices to show that the subset $A_0 = A \cap
  F_0$ contained in the facet $F_0 = (0,\emptyset,\dotsc,\emptyset)$ is affinely independent for every realization of
  $\wp{p}{1}$.

  Consider the following flag $\bar{v}_{0} = G_0 \subset G_1 \subset \dots \subset G_{p+1}=F_0$ of faces $G_i
  \subset \wp{p}{1}$ with:
  \begin{align*}
    G_0 &= ([2],[2],0,\dots,0), \qquad
    G_1 = ([2],0,0,\dots,0), \qquad
    G_2 = (0,0,0,\dots,0), \text{ and} \\
    G_k &= (0,0,0,\dots,0,\underbrace{\emptyset,\dots,\emptyset}_{k-2}) \quad \text{for $k=3,\dots,p+1$.}
  \end{align*}
  Then $\dim G_i = i$ and $|A_0 \cap G_i| = i+1$ and thus $A_0$ is affinely independent.
\end{proof}

Using Corollary~\ref{cor:ModuliProjection} we obtain the following lower bound on the number of moduli of the wedge
product surfaces.

\begin{corollary}[Moduli of $\WPsurf{p}{4}$] 
  The realizations of the surfaces~$\WPsurf{p}{4} \subset \R^3$ obtained via projections of the wedge products
  $\wp{p}{1} \subset \R^{2+p}$ have at least $6p$ moduli.
\end{corollary}

\begin{remark}
  The duals $\WPsurf{p}{4}^*$ of the surfaces $\WPsurf{p}{4}$ are also contained in high-dimensional cubes and realizations are obtained
  via projections in~\cite{JoswigRoerig2007:NeighborlyCubical}. Using affine support sets of cubes yields the same lower
  bounds on the number of moduli for the dual surfaces (cf.~\cite[Thm.~2.33]{Thilo-phd}) even though the construction is
  based on a different high-dimensional polytope.
\end{remark}

\bibliographystyle{amsplain}
\begin{small}
    \bibliography{literature}
\end{small}

\end{document}